\documentclass{amsart}

\usepackage{hyperref, amssymb, amsmath, amscd}
\usepackage[all]{xy}
\usepackage{paralist, enumerate}

\allowdisplaybreaks[4]

\theoremstyle{plain}
\newtheorem{theorem}{Theorem}[section]
\newtheorem{proposition}[theorem]{Proposition}
\newtheorem{lemma}[theorem]{Lemma}
\newtheorem{corollary}[theorem]{Corollary}

\theoremstyle{definition}
\newtheorem{remark}[theorem]{Remark}

\newcommand{\Affine}{\mathbb{A}}
\newcommand{\CC}{\mathbb{C}}
\newcommand{\PP}{\mathbb{P}}
\newcommand{\ZZ}{\mathbb{Z}}

\newcommand{\I}{\mathcal{I}}

\newcommand{\J}{\mathcal{J}}

\newcommand{\Osh}{\mathcal{O}}

\newcommand{\q}{\mathfrak{q}}

\newcommand{\BN}{\mathrm{BN}}
\newcommand{\Hilb}{\mathrm{Hilb}}
\newcommand{\Hor}{\mathrm{Hor}}
\newcommand{\pr}{\mathrm{pr}}
\newcommand{\red}{\mathrm{red}}
\newcommand{\Ver}{\mathrm{Ver}}

\newcommand{\bip}{\operatorname{b}}
\newcommand{\csi}{\operatorname{\xi}}
\newcommand{\euler}{\operatorname{\chi}}
\newcommand\HH{\operatorname{H}}
\newcommand{\length}{\operatorname{length}}
\newcommand{\Spec}{\operatorname{Spec}}

\newcommand{\lra}{\longrightarrow}

\begin{document}

\setdefaultleftmargin{21pt}{}{}{}{}{}

\title[On the Hilbert Function of a Finite Scheme]
{On the Hilbert Function of a Finite Scheme Contained in a Quadric Surface}

\author{Mario Maican}
\address{Institute of Mathematics of the Romanian Academy, Calea Grivitei 21, Bucharest 010702, Romania}

\email{maican@imar.ro}

\begin{abstract}
Consider a finite scheme of length $l$ contained in a smooth quadric surface over the complex numbers.
We determine the number of linearly independent curves passing through the scheme, of degree at least $l - 1$.
\end{abstract}

\subjclass[2010]{Primary 14C05, 14F17}
\keywords{Topological Euler characteristic, Flag Hilbert schemes, Brill-Noether loci}

\maketitle

\section{Introduction}
\label{introduction}

\noindent
All schemes considered in this paper will be over $\CC$.
We shall work on the smooth quadric surface $\PP^1 \times \PP^1$.
We choose homogeneous coordinates $(x : y)$ on the first $\PP^1$ and $(z : w)$ on the second $\PP^1$.
Given a positive integer $l$, we denote by $\Hilb(l)$ the Hilbert scheme of $l$ points in $\PP^1 \times \PP^1$.
Given non-negative integers $m$ and $n$, which are not both zero, we denote by $\Hilb(l, (m, n))$
the flag Hilbert scheme parametrizing pairs $(Z, C)$,
where $C \subset \PP^1 \times \PP^1$ is a curve of bidegree $(m, n)$
and $Z \subset C$ is a finite subscheme of length $l$.
Given $Z \in \Hilb(l)$ we define its Hilbert function by analogy
with the Hilbert function of a finite subscheme of the projective space, namely
\[
\ZZ_{\ge 0}^{} \times \ZZ_{\ge 0}^{} \ni (m, n) \longmapsto \dim_\CC^{} \HH^0(\Osh(m, n)) - \dim_\CC^{} \HH^0(\I_Z^{}(m, n)).
\]
In this paper we determine the values of the Hilbert function of $Z$ for $m + n \ge l - 1$, see Theorem~\ref{BN}.
As a consequence, we compute the topological Euler characteristic of $\Hilb(l, (m, n))$, see Theorem~\ref{HE}.
The answer is given in terms of the Euler characteristic of $\Hilb(l)$,
which was computed in \cite{ellingsrud},
and in terms of the numbers $\csi(l, m)$ and $\csi(l, n)$,
which can be computed by an easy algorithm.
The proof of Theorem~\ref{BN} rests on the Vanishing Theorem~\ref{vanishing},
which relies on the fact that there are curves of arithmetic genus zero and arbitrary degree in $\PP^1 \times \PP^1$.
For $Z$ contained in such a curve,
an analysis can be performed, as in section~\ref{vanishing_theorem}.

The paper is organized as follows.
In section~\ref{line} we recall a few basic facts concerning the cohomology of a finite subscheme of a multiple line.
In section~\ref{punctual}, following Ellingsrud and Str{\o}mme \cite{ellingsrud},
we review the geometry of the punctual Hilbert scheme of a surface.
This will be needed in section~\ref{euler},
which is devoted to the calculation of the Euler characteristic of the flag Hilbert scheme.
Sections~\ref{vanishing_theorem} and \ref{loci} are occupied by the Vanishing Theorem, respectively,
by the description of the Brill-Noether loci in $\Hilb(l)$,
i.e.\ the loci on which the dimension of the fibers of the forgetful morphism
$\Hilb(l, (m, n)) \to \Hilb(l)$ jumps.

\section{Finite schemes contained in a multiple line}
\label{line}

\noindent
In this section we collect several well-known facts about finite subschemes of a multiple projective line.
Let $\nu \ge 1$ be an integer.
We denote by $\nu L$ the multiple projective line $\Spec \CC[z]/ (z^\nu) \times \PP^1$.
Given an integer $m$, we write $\Osh_{\nu L}^{}(m) = \pr_2^* \Osh_{\PP^1}^{}(m)$.

\begin{proposition}
\label{ZL}
Let $Z \subset \nu L$ be a zero-dimensional subscheme.
Consider the restricted scheme $Z_0^{} = Z \cap L$.
We claim that the ideal sheaf of $Z_0^{}$ in $Z$ is isomorphic, as an $\Osh_Z^{}$-module,
to the structure sheaf of a finite scheme $Z' \subset (\nu - 1) L$.
\end{proposition}

\noindent
Let $Z \subset \nu L$ be a finite subscheme.
By induction on $\nu$, we define a string of non-negative integers $\rho(Z, \nu L) = (r_0^{}, \dots, r_{\nu - 1}^{})$, as follows.
If $\nu = 1$, we put $r_0^{} = \length(Z)$.
If $\nu \ge 2$, we consider the schemes $Z_0^{}$ and $Z'$ from Proposition~\ref{ZL}.
We write $\rho(Z', (\nu - 1)L) = (r_0', \dots, r_{\nu - 2}')$
and we let $r_0^{} = \length(Z_0^{})$ and $r_i^{} = r_{i - 1}'$ for $1 \le i \le \nu - 1$.

\begin{remark}
\label{rho}
Clearly, $r_0^{} + \dots + r_{\nu - 1}^{} = \length(Z)$ and $r_0^{} \ge \dots \ge r_{\nu - 1}^{}$.
\end{remark}

\begin{proposition}
\label{HZL}
Let $\nu \ge 1$ and $m$ be integers.
Let $Z \subset \nu L$ be a zero-dimensional scheme.
Let $\rho(Z, \nu L) = (r_0^{}, \dots, r_{\nu - 1}^{})$ be as above.
We claim that
\[
\dim_\CC^{} \HH^1(\I_{Z, \nu L}^{}(m)) = \sum_{\substack{0 \le i \le \nu - 1 \\ m + 1 < r_i}} (r_i^{} - m - 1).
\]
\end{proposition}

\noindent
Let $Z \subset \PP^1 \times \PP^1$ be a zero-dimensional scheme.
Let $\mu(Z)$ be the smallest integer $\mu$ such that there exists a curve $M \subset \PP^1 \times \PP^1$
of bidegree $(\mu, 0)$ containing $Z$.
Let $\nu(Z)$ be the smallest integer $\nu$ such that there exists a curve $N \subset \PP^1 \times \PP^1$
of bidegree $(0, \nu)$ containing $Z$.
Note that $M$ and $N$ are unique, so they may be denoted by $M(Z)$, respectively, $N(Z)$.
Let $D \subset \PP^1 \times \PP^1$ be a line of bidegree $(1, 0)$.
If $M = \mu D$, then $\rho(Z, \mu D)$ depends only on $Z$, so it may be denoted by $\sigma(Z)$.
Let $E \subset \PP^1 \times \PP^1$ be a line of bidegree $(0, 1)$.
If $N = \nu E$, then $\rho(Z, \nu E)$ depends only on $Z$, so it may be denoted by $\tau(Z)$.
Note that $\sigma(Z)$ and $\tau(Z)$ are strings of positive integers.
Let $\kappa(Z)$ be the number of distinct lines in the reduced support of $M(Z)$.
Write
\[
M(Z) = \bigcup_{1 \le i \le \kappa(Z)} \mu_i^{} D_i^{}.
\]
Let $X_i^{}$ be the subscheme of $Z$ that is concentrated on $D_i^{}$.
Let $\lambda(Z)$ be the number of distinct lines in the reduced support of $N(Z)$.
Write
\[
N(Z) = \bigcup_{1 \le j \le \lambda(Z)} \nu_j^{} E_j^{}.
\]
Let $Y_j^{}$ be the subscheme of $Z$ that is concentrated on $E_j^{}$.

\begin{corollary}
\label{HZM}
Let $Z \subset \PP^1 \times \PP^1$ be a zero-dimensional scheme.
Let $m \ge \mu(Z) - 1$ and $n \ge -1$ be integers.
We adopt the above notations.
\begin{enumerate}
\item[\emph{(i)}]
We claim that ${\displaystyle \dim_\CC^{} \HH^1(\I_Z^{}(m, n))
= \sum_{1 \le i \le \kappa(Z)} \sum_{\substack{0 \le j \le \mu_i - 1 \\ \sigma(X_i)_j > n + 1}} (\sigma(X_i^{})_j^{} - n - 1)}$.
\item[\emph{(ii)}]
Assume, in addition, that $\length(Z \cap D) \le n + 1$ for every line $D \subset \PP^1 \times \PP^1$ of bidegree $(1, 0)$.
We claim that $\HH^1(\I_Z^{}(m, n)) = \{ 0 \}$.
\end{enumerate}
\end{corollary}

\begin{proof}
(i) From the exact sequence
\[
0 \lra \Osh(m - \mu(Z), n) \lra \I_Z^{}(m, n) \lra \I_{Z, M(Z)}^{}(n) \lra 0,
\]
and from the vanishing of $\HH^q(\Osh(m - \mu(Z), n))$ for $q = 1$, $2$,
we get the isomorphism
\[
\HH^1(\I_Z^{}(m, n)) \simeq \HH^1(\I_{Z, M(Z)}(n)).
\]
The claim follows from Proposition~\ref{HZL}.

\medskip

\noindent
(ii) By hypothesis, for all indices $1 \le i \le \kappa(Z)$, we have the inequality
$\sigma(X_i^{})_0^{} = \length(Z \cap D_i^{}) \le n + 1$.
In view of Remark~\ref{rho}, for $1 \le j \le \mu_i^{} - 1$, we have the inequalities
$\sigma(X_i^{})_j^{} \le \sigma(X_i^{})_0^{} \le n + 1$.
The r.h.s.\ in formula (i) vanishes.
\end{proof}

\begin{corollary}
\label{HZN}
Let $Z \subset \PP^1 \times \PP^1$ be a zero-dimensional scheme.
Let $n \ge \nu(Z) - 1$ and $m \ge -1$ be integers.
We adopt the above notations.
\begin{enumerate}
\item[\emph{(i)}] We claim that ${\displaystyle \dim_\CC^{} \HH^1(\I_Z^{}(m, n))
= \sum_{1 \le j \le \lambda(Z)} \sum_{\substack{0 \le i \le \nu_j - 1 \\ \tau(Y_j)_i > m + 1}} (\tau(Y_j^{})_i^{} - m - 1)}$.
\item[\emph{(ii)}]
Assume, in addition, that $\length(Z \cap E) \le m + 1$ for every line $E \subset \PP^1 \times \PP^1$ of bidegree $(0, 1)$.
We claim that $\HH^1(\I_Z^{}(m, n)) = \{ 0 \}$.
\end{enumerate}
\end{corollary}

\section{An affine decomposition of the punctual Hilbert scheme}
\label{punctual}

\noindent
Let $p \in \PP^1 \times \PP^1$ be the point given by the equations $x = 0$, $z = 0$.
Let $l$ be a positive integer.
Let $\Hilb_p^{}(l)$ be the punctual Hilbert scheme parametrizing subschemes of length $l$ of $\PP^1 \times \PP^1$
that are concentrated at $p$.
In this section we will give a decomposition of $\Hilb_p^{}(l)$ into locally closed subsets
that are isomorphic to affine spaces.
This decomposition is due to Ellingsrud and Str{\o}mme, see \cite{ellingsrud}.
The following proposition is well-known.

\begin{proposition}
\label{IP}
Let $P \subset \PP^1 \times \PP^1$ be a subscheme that is concentrated at $p$.
Write $\nu = \nu(P)$ and $\tau(P) = (t_0^{}, \dots, t_{\nu - 1}^{})$.
Let $I_P^{} \subset \Osh_p^{} = \CC[x, z]_{(x, z))}^{}$ be the ideal of $P$.
We claim that, for $0 \le k \le \nu - 2$, there are polynomials of the form
\[
f_k^{} = x^{t_k} z^k + \sum_{j = k + 1}^{\nu - 1} \sum_{i = 0}^{t_j - 1} a_{ijk}^{} \, x^i z^j
\]
such that $I_P^{} = (f_0^{}, \dots, f_{\nu - 2}^{}, \, x^{t_{\nu - 1}} z^{\nu - 1}, \, z^\nu)$.
We claim that the set of polynomials
\[
\{ x^i z^j \mid 0 \le j \le \nu - 1, \ 0 \le i \le t_j^{} - 1 \}
\]
gives a basis of $\Osh_{P, p}^{} = \Osh_p^{}/ I_P^{}$ as a $\CC$-vector space.
\end{proposition}

\begin{corollary}
\label{OP}
We adopt the above notations.
We claim that the set of polynomials
\[
\{ x^i z^j \mid 0 \le i \le t_0^{} - 1, \ 0 \le j \le \nu - 1 \}
\]
generates $\Osh_{P, p}^{}$ as a $\CC$-vector space.
\end{corollary}

\noindent
Indeed, according to Remark~\ref{rho}, $t_j^{} \le t_0^{}$,
so the set of polynomials from the above corollary contains the basis from the above proposition.
Given a positive integer $l$, we consider the set of partitions of $l$,
\[
\Pi(l) = \{ \tau = (t_0^{}, \dots, t_{\nu - 1}^{}) \mid t_0^{} \ge \dots \ge t_{\nu - 1}^{} > 0, \ t_0^{} + \dots + t_{\nu - 1}^{} = l \}.
\]
Given $\tau \in \Pi(l)$, we consider the subset $A(\tau) = \{ P \mid \tau(P) = \tau \} \subset \Hilb_p^{}(l)$.

\begin{proposition}
\label{affine}
The family $\{ A(\tau) \}_{\tau \in \Pi(l)}^{}$ constitutes a decomposition of $\Hilb_p^{}(l)$ into locally closed subsets.
We equip $A(\tau)$ with the induced reduced structure.
We claim that $A(\tau) \simeq \Affine^{l - t_0}$.
\end{proposition}

\noindent
The above proposition is a direct consequence of the proof of \cite[Theorem (1.1)(iv)]{ellingsrud}.
We will sketch the argument here for the sake of the reader.
Consider the action of the multiplicative group $\CC^*$ on $\CC[x, z]$ given by $a. x^i z^j = a^j x^i z^j$.
Identifying $\Hilb_p^{}(l)$ with the set of ideals in $\CC[x, z]$ that have colenght $l$ and that are contained in $(x, z)$,
we obtain an induced action of $\CC^*$ on $\Hilb_p^{}(l)$.
The set of fixed points for this action is $\{ P_\tau^{} \mid \tau \in \Pi(l) \}$,
where $P_\tau^{}$ is given by the ideal $(x^{t_0}, \, x^{t_1} z, \dots, x^{t_{\nu - 1}} z^{\nu - 1}, \, z^\nu)$.
Choose an arbitrary $P \in A(\tau)$
and write $I_P^{} = (f_0^{}, \dots, f_{\nu - 2}^{}, \, x^{t_{\nu - 1}} z^{\nu - 1}, \, z^\nu)$, as in Proposition~\ref{IP}.
Note that ${\displaystyle \lim_{a \to 0} a. P}$ has ideal
\begin{align*}
\lim_{a \to 0} & \Big( \dots, a^k x^{t_k} z^k + \sum_{j = k + 1}^{\nu - 1} \sum_{i = 0}^{t_j - 1} a_{ijk}^{} \, a^j x^i z^j, \dots,
a^{\nu - 1} x^{t_{\nu - 1}} z^{\nu - 1}, \, a^\nu z^\nu \Big) \\
& = \lim_{a \to 0} \Big( \dots, x^{t_k} z^k + \sum_{j = k + 1}^{\nu - 1} \sum_{i = 0}^{t_j - 1} a_{ijk}^{} \, a^{j - k} x^i z^j, \dots,
x^{t_{\nu - 1}} z^{\nu - 1}, \, z^\nu \Big) \\
& = (\dots, x^{t_k} z^k, \dots, x^{t_{\nu - 1}} z^{\nu - 1}, \, z^\nu),
\end{align*}
which is the ideal of $P_\tau^{}$.
Thus, ${\displaystyle \lim_{a \to 0} a. P = P_\tau^{}}$.
It has now become clear that
\[
A(\tau) = \{ P \in \Hilb_p^{}(l) \mid \lim_{a \to 0} a. P = P_\tau^{} \},
\]
so we may apply \cite[Theorem 4.4]{birula}
in order to deduce that $\{ A(\tau) \}_{\tau \in \Pi(l)}^{}$ is a locally closed decomposition of $\Hilb_p^{}(l)$
and that each $A(\tau)$ is isomorphic to an affine space.
The dimension of these affine spaces was computed in \cite[p.\ 350]{ellingsrud}.

\section{The vanishing theorem}
\label{vanishing_theorem}

\noindent
In this section we prove the Vanishing Theorem~\ref{vanishing},
which will lead us to the description of the Brill-Noether loci in Theorem~\ref{BN}.
The key technical step towards the proof of Theorem~\ref{vanishing} is contained in the following lemma.

\begin{lemma}
\label{ZCE}
Consider integers $c \ge 0$, $\nu \ge 1$, $m \ge 1$ and $n \ge c + \nu - 1$.
Let $C \subset \PP^1 \times \PP^1$ be a curve of bidegree $(1, c)$.
Let $E \subset \PP^1 \times \PP^1$ be a line of bidegree $(0, 1)$, which is not contained in $C$.
Consider a zero-dimensional scheme $Z \subset C \cup \nu E$ and consider its restriction $X = Z \cap C$.
Assume that $\length(Z \cap E) \le m$.
Then
\[
\HH^1(\I_Z^{}(m, n)) \simeq \HH^1 (\I_X^{}(m, n)).
\]
\end{lemma}

\begin{proof}
Consider the ideal $\I_{X, Z}^{} \subset \Osh_Z^{}$ of $X$ in $Z$.
From the short exact sequence
\[
0 \lra \I_Z^{}(m, n) \lra \I_X^{}(m, n) \lra \I_{X, Z}^{} \lra 0
\]
we obtain the long exact cohomology sequence
\[
\HH^0(\I_X^{}(m, n)) \lra \HH^0(\I_{X, Z}^{})
\lra \HH^1(\I_Z^{}(m, n)) \lra \HH^1(\I_X^{}(m, n)) \lra \HH^1(\I_{X, Z}^{}).
\]
The space on the right vanishes because $\I_{X, Z}^{}$ is supported on a finite set.
We have reduced the lemma to proving that the first arrow is surjective.
We may assume that $Z$ is contained in the affine chart $U = \{ y \neq 0, \, w \neq 0 \}$.
We may assume that $C$ does not contain the line given by the equation $w = 0$.
We may further assume that the point $p = C \cap E$ is given by the equations $x = 0$, $z = 0$.
Choose a polynomial $f(x, z)$ of degree $1$ in the variable $x$ and of degree $c$ in the variable $z$,
which vanishes on $C \cap U$.
Let $P$ be the (possibly empty) subscheme of $Z$ that is concentrated on $p$.
Denote $e = \length(P \cap E)$.
Denote $Z_\red^{} \cap E \setminus \{ p \} = \{ p_1^{}, \dots, p_r^{} \}$.
For $1 \le k \le r$, let $P_k^{}$ be the subscheme of $Z$ that is concentrated on $p_k^{}$.
Let $\J = f \Osh_P^{} \subset \Osh_P^{}$ be the ideal sheaf of $P \cap C$ in $P$.
Note that
\[
\I_{X, Z}^{} = \J \oplus \Osh_{P_1}^{} \oplus \dots \oplus \Osh_{P_r}^{}.
\]
Let $\q \subset \Osh_p^{} = \CC[x, z]_{(x, z)}^{}$ be the annihilator of $\J_p^{}$.
Let $Q \subset \PP^1 \times \PP^1$ be the finite scheme concentrated at $p$ that is defined by $\q$.
By hypothesis, $z^\nu f = 0$ in $\Osh_P^{}$, hence $z^\nu \in \q$, and hence $Q \subset \nu E$.
According to Corollary~\ref{OP}, the set of polynomials
\[
\{ f x^i z^j \mid 0 \le i \le e - 1, \ 0 \le j \le \nu(P) - 1 \}
\]
generates $\J_p^{}$ as a $\CC$-vector space.
It follows that the set of polynomials
\[
\{ x^i z^j \mid 0 \le i \le e - 1, \ 0 \le j \le \nu(P) - 1 \}
\]
generates $\Osh_{Q, p}^{}$ as a $\CC$-vector space,
hence $\{ x^i \mid 0 \le i \le e - 1 \}$ generates $\Osh_{Q \cap E, p}^{}$,
and hence $\length(Q \cap E) \le e$.
The finite scheme $Y = Q \cup P_1^{} \cup \dots \cup P_r^{}$ is contained in $\nu E$.
It satisfies the condition $\length(Y \cap E) \le \length(Z \cap E) \le m$.
According to Corollary~\ref{HZN}(ii), $\HH^1(\I_Y^{}(m - 1, \nu - 1)) = \{ 0 \}$.
It follows that the set of polynomials\[
\{ x^i z^j \mid 0 \le i \le m - 1, \ 0 \le j \le \nu - 1 \}
\]
generates $\HH^0(\Osh_Y^{})$ as a $\CC$-vector space.
Since $f$ does not vanish at $p_1^{}, \dots, p_r^{}$,
we deduce that the set of polynomials
\[
G = \{ f x^i z^j \mid 0 \le i \le m - 1, \ 0 \le j \le \nu - 1 \}
\]
generates $\HH^0(f \Osh_Y^{}) = \HH^0(\I_{X, Z}^{})$ as a $\CC$-vector space.
Each $g \in G$ vanishes on $C \cap U$, so it vanishes on $X$.
The degree of $g$ in the variable $x$ is at most $m$.
The degree of $g$ in the variable $z$ is at most $c + \nu - 1 \le n$.
Given $g \in G$, we consider the form $\bar{g} = g(\frac{x}{y}, \frac{z}{w}) y^m w^n \in \HH^0(\I_X^{}(m, n))$.
The set $\{ \bar{g} \mid g \in G \}$ maps to a set of generators of $\HH^0(\I_{X, Z}^{})$ as a $\CC$-vector space.
We deduce that the map $\HH^0(\I_X^{}(m, n)) \to \HH^0(\I_{X, Z}^{})$ is surjective,
which concludes the proof of the lemma.
\end{proof}

\begin{remark}
\label{HZC}
Let $C \subset \PP^1 \times \PP^1$ be an irreducible curve of bidegree $(1, c)$.
Let $Z \subset C$ be a finite scheme of length $l$.
Let $m \ge 0$ and $n \ge c - 1$ be integers.
Then
\[
\HH^1(\I_Z^{}(m, n)) \simeq \HH^1(\Osh_{\PP^1}^{}(c^{} m + n - l)).
\]
Indeed, from the exact sequence
\[
0 \lra \I_C^{}(m, n) \simeq \Osh(m - 1, n - c) \lra \I_Z^{}(m, n) \lra \I_{Z, C}^{}(m, n) \lra 0,
\]
and from the vanishing of $\HH^q(\Osh(m - 1, n - c))$ for $q = 1$, $2$,
we obtain the isomorphism
\[
\HH^1(\I_Z^{}(m, n)) \simeq \HH^1(\I_{Z, C}^{}(m, n)).
\]
Note that $\I_{Z, C}^{}(m, n)$ is supported on $C \simeq \PP^1$
and $\I_{Z, C}^{}(m, n) |_C \simeq \Osh_{\PP^1}^{}(c^{} m + n - l)$.
\end{remark}

\begin{proposition}
\label{ZC}
Let $C \subset \PP^1 \times \PP^1$ be a curve of bidegree $(1, c)$.
Let $Z \subset C$ be a finite scheme of length $l$.
Let $m \ge 1$ and $n \ge c - 1$ be integers such that $m + n \ge l - 1$.
Assume that $\length(Z \cap D) \le n + 1$ for every line $D \subset \PP^1 \times \PP^1$ of bidegree $(1, 0)$.
Assume that $\length(Z \cap E) \le m$ for every line $E \subset \PP^1 \times \PP^1$ of bidegree $(0, 1)$.
Then
\[
\HH^1(\I_Z^{}(m, n)) = \{ 0 \}.
\]
\end{proposition}

\begin{proof}
We proceed by induction on the number of irreducible components of $C$.
If $C$ is irreducible, then, by virtue of Remark~\ref{HZC},
\[
\HH^1(\I_Z^{}(m, n)) \simeq \HH^1(\Osh_{\PP^1}^{}(c^{} m + n - l)).
\]
The r.h.s.\ vanishes because $c^{} m + n - l \ge -1$.
Indeed, if $c = 0$, then this inequality holds by hypothesis.
If $c > 0$, then $c^{} m + n - l \ge m + n - l \ge -1$.
Assume now that $C = C' \cup \nu E$,
where $C'$ is a curve of bidegree $(1, c')$ which does not contain $E$, and $\nu \ge 1$.
Denote $X = Z \cap C'$.
According to Lemma~\ref{ZCE},
\[
\HH^1(\I_Z^{}(m, n)) \simeq \HH^1 (\I_X^{}(m, n)).
\]
Since $C'$ has fewer irreducible components than $C$,
we may apply the induction hypothesis to $X \subset C'$
in order to deduce that the r.h.s.\ vanishes.
\end{proof}

\begin{theorem}
\label{vanishing}
Let $Z \subset \PP^1 \times \PP^1$ be a finite scheme of length $l$.
Let $m$ and $n$ be non-negative integers such that $m + n \ge l - 1$.
Assume that $\length(Z \cap D) \le n + 1$ for every line $D \subset \PP^1 \times \PP^1$ of bidegree $(1, 0)$.
Assume that $\length(Z \cap E) \le m + 1$ for every line $E \subset \PP^1 \times \PP^1$ of bidegree $(0, 1)$.
Then $\HH^1(\I_Z^{}(m, n)) = \{ 0 \}$.
\end{theorem}

\begin{proof}
By symmetry, we may assume that $n \ge m$.
If $m = 0$, then $n \ge l - 1 \ge \nu(Z) - 1$,
and the conclusion follows from Corollary~\ref{HZN}(ii).
Assume that $m \ge 1$.
If $\length(Z \cap E) = m + 1$ for some line $E \subset \PP^1 \times \PP^1$ of bidegree $(0, 1)$,
then $\nu(Z) \le l - m \le n + 1$, and we again apply Corollary~\ref{HZN}(ii).
Assume that  $\length(Z \cap E) \le m$ for every line $E \subset \PP^1 \times \PP^1$ of bidegree $(0, 1)$.
Write $l = 2 c$ or $l = 2 c + 1$ for an integer $c$.
Notice that $n \ge c$ and that $Z$ is contained in a curve of bidegree $(1, c)$.
The conclusion follows from Proposition~\ref{ZC}.
\end{proof}

\section{The Brill-Noether loci}
\label{loci}

\noindent
Consider integers $k \ge 0$, $l \ge 1$, $m \ge 0$ and $n \ge 0$ such that $m$ and $n$ are not both zero.
Consider the Brill-Noether locus
\[
\BN(k, l, (m, n)) = \{ Z \mid \dim_\CC^{} \HH^1(\I_Z^{}(m, n)) = k \} \subset \Hilb(l).
\]
By the semicontinuity theorem, this locus is locally closed,
so, for any $m$ and $n$ as above, we have a locally closed decomposition
\begin{equation}
\label{decomposition}
\Hilb(l) = \bigsqcup_{k \ge 0} \BN(k, l, (m, n)).
\end{equation}
In what follows we adopt the notations of section~\ref{line}.
Assume that $k \ge 1$.
Inside $\Hilb(l)$ we consider the subset $S(k, l, n)$ of subschemes $Z$ satisfying the equation
\begin{equation}
\label{vertical}
k = \sum_{1 \le i \le \kappa(Z)} \sum_{\substack{0 \le j \le \mu_i - 1 \\ \sigma(X_i)_j > n + 1}} (\sigma(X_i^{})_j^{} - n - 1),
\end{equation}
and the subset $T(k, l, m)$ of subschemes $Z$ satisfying the equation
\begin{equation}
\label{horizontal}
k = \sum_{1 \le j \le \lambda(Z)} \sum_{\substack{0 \le i \le \nu_j - 1 \\ \tau(Y_j)_i > m + 1}} (\tau(Y_j^{})_i^{} - m - 1).
\end{equation}
Inside $\Hilb(l)$ we consider the subset $\Ver(k, l)$ of schemes $Z$
for which there exists a line $D \subset \PP^1 \times \PP^1$ of bidegree $(1, 0)$ (a vertical line) such that $\length(Z \cap D) = k$.
Inside $\Hilb(l)$ we consider the subset $\Hor(k, l)$ of schemes $Z$
for which there exists a line $E \subset \PP^1 \times \PP^1$ of bidegree $(0, 1)$ (a horizontal line) such that $\length(Z \cap E) = k$.
Note that $\Ver(l, l)$ is the closed subvariety of $\Hilb(l)$ of schemes that are contained in a vertical line,
while $\Hor(l, l)$ is the closed subvariety of $\Hilb(l)$ of schemes that are contained in a horizontal line.
Both subvarieties are isomorphic to $\PP^1 \times \PP^{}{}^l$.

\begin{proposition}
\label{S}
Consider integers $l \ge 1$ and $n \ge 0$.
\begin{enumerate}
\item[\emph{(i)}]
If $l \le n + 1$, then $S(k, l, n) = \emptyset$ for all $k \ge 1$.
If $l > n + 1$, then $S(k, l, n)$ is nonempty if and only if $1 \le k \le l - n - 1$.
\item[\emph{(ii)}]
Assume that $l > n + 1$.
Then $S(l - n - 1, l, n) = \Ver(l, l)$.
\item[\emph{(iii)}]
Assume that $l \le 2 n + 3$.
Then $S(k, l, n) = \Ver(k + n + 1, l)$ for all $k \ge 1$.
\end{enumerate}
\end{proposition}

\begin{proof}
(i) Assume that $k \ge 1$ and $Z \in S(k, l, n)$.
From equation~\eqref{vertical} we deduce that there are indices $i$ and $j$ such that $\sigma(X_i^{})_j^{} > n + 1$.
From Remark~\ref{rho} we obtain the inequality $\sigma(X_i^{})_0^{} \ge \sigma(X_i^{})_j^{} > n + 1$.
Equation~\eqref{vertical} can be rewritten in the form
\[
k = \sigma(X_i^{})_0^{} - n - 1 + \sum_{\substack{1 \le j \le \mu_i - 1 \\ \sigma(X_i)_j > n + 1}} \!\!\!\!\!\!\! (\sigma(X_i^{})_j^{} - n - 1)
+ \sum_{\substack{1 \le h \le \kappa(Z) \\ h \neq i}} \sum_{\substack{0 \le j \le \mu_h - 1 \\ \sigma(X_h)_j > n + 1}}
\!\!\!\!\!\!\! (\sigma(X_h^{})_j^{} - n - 1).
\]
This leads to the inequality
\begin{equation}
\label{strict}
k \le \sigma(X_i^{})_0^{} - n - 1 + \sum_{1 \le j \le \mu_i - 1} \sigma(X_i^{})_j^{}
+ \sum_{\substack{1 \le h \le \kappa(Z) \\ h \neq i}} \sum_{0 \le j \le \mu_h - 1} \sigma(X_h^{})_j^{}.
\end{equation}
According to Remark~\ref{rho}, the r.h.s.\ equals $l - n - 1$.
Thus, if $S(k, l, n)$ is nonempty, then $k \le l - n - 1$ forcing $l > n + 1$.
Assume now that $1 \le k \le l - n - 1$.
Consider distinct lines $D_1^{}, \dots, D_{l - k - n}^{}$ of bidegree $(1, 0)$ in $\PP^1 \times \PP^1$.
Let $X_1^{} \subset D_1^{}$ be a subscheme of length $k + n + 1$.
For $2 \le i \le l - k - n$, let $X_i^{} \subset D_i^{}$ be a point.
Put $Z = X_1^{} \cup \dots \cup X_{l - k - n}^{}$ and notice that $Z \in S(k, l, n)$.

\medskip

\noindent
(ii) Note that inequality~\eqref{strict} is strict if $\mu(Z) > 1$.
For $Z \in S(l - n - 1, l, n)$, inequality~\eqref{strict} is not strict, hence $\mu(Z) = 1$,
that is, $Z$ is contained in a line of bidegree $(1, 0)$.
This proves the inclusion $S(l - n - 1, l, n) \subset \Ver(l, l)$.
The reverse inclusion is obvious.

\medskip

\noindent
(iii) Assume that $Z \in S(k, l, n)$.
As noted above, there is an index $i$ such that $\sigma(X_i^{})_0^{} > n + 1$.
For all indices $1 \le j \le \mu_i^{} - 1$ we have the inequality $\sigma(X_i^{})_j^{} \le n + 1$;
for all indices $h \neq i$, $1 \le h \le \kappa(Z)$, and $j$ we have the inequality $\sigma(X_h^{})_j^{} \le n + 1$,
otherwise $l \ge 2 n + 4$.
Equation~\eqref{vertical} takes the form $k = \sigma(X_i^{})_0^{} - n - 1$, that is, $\length(Z \cap D_i^{}) = k + n + 1$.
Thus, $Z \in \Ver(k + n + 1, l)$.
This proves the inclusion $S(k, l, n) \subset \Ver(k + n + 1, l)$.
The reverse inclusion can be proved analogously.
\end{proof}

\begin{proposition}
\label{T}
Consider integers $l \ge 1$ and $m \ge 0$.
\begin{enumerate}
\item[\emph{(i)}]
If $l \le m + 1$, then $T(k, l, m) = \emptyset$ for all $k \ge 1$.
If $l > m + 1$, then $T(k, l, m)$ is nonempty if and only if $1 \le k \le l - m - 1$.
\item[\emph{(ii)}]
Assume that $l > m + 1$.
Then $T(l - m - 1, l, m) = \Hor(l, l)$.
\item[\emph{(iii)}]
Assume that $l \le 2 m + 3$.
Then $T(k, l, m) = \Hor(k + m + 1, l)$ for all $k \ge 1$.
\end{enumerate}
\end{proposition}

\begin{theorem}
\label{BN}
Let $k \ge 1$, $l \ge 2$, $m \ge 0$ and $n \ge 0$ be integers such that $m + n \ge l - 1$.
Then we have the locally closed decomposition
\[
\BN(k, l, (m, n)) = S(k, l, n) \sqcup T(k, l, m).
\]
\end{theorem}

\begin{proof}
Take $Z \in \BN(k, l, (m, n))$.
According to Theorem~\ref{vanishing},
$\length(Z \cap D) \ge n + 2$ for some line $D \subset \PP^1 \times \PP^1$ of bidegree $(1, 0)$,
or $\length(Z \cap E) \ge m + 2$ for some line $E \subset \PP^1 \times \PP^1$ of bidegree $(0, 1)$.
In the first case, $\mu(Z) \le l - n - 1 \le m$,
so we may apply Corollary~\ref{HZM}(i) in order to obtain formula~\eqref{vertical}.
In the second case, $\nu(Z) \le l - m - 1\le n$,
so we may apply Corollary~\ref{HZN}(i) in order to obtain formula~\eqref{horizontal}.
This proves the inclusion ``$\subset$''.
Take $Z \in S(k, l, n)$.
As mentioned in the proof of Proposition~\ref{S}(i),
$\length(Z \cap D_i^{}) = \sigma(X_i^{})_0^{} \ge n + 2$ for an index $i$.
As noted above, this inequality allows us to apply Corollary~\ref{HZM}(i).
We deduce that $\dim_\CC^{} \HH^1(\I_Z^{}(m, n)) = k$.
Assume that $Z \in T(k, l, m)$.
There is an index $j$ such that $\length(Z \cap E_j^{}) = \tau(Y_j^{})_0^{} \ge m + 2$.
As noted above, this inequality allows us to apply Corollary~\ref{HZN}(i).
We deduce that $\dim_\CC^{} \HH^1(\I_Z^{}(m, n)) = k$.
This proves the reverse inclusion.
The union is disjoint because we cannot have the inequalities
$\length(Z \cap D_i^{}) \ge n + 2$ and $\length(Z \cap E_j^{}) \ge m + 2$ at the same time.
Indeed,
\[
m + n + 2 \ge l + 1 \ge \length(Z \cap (D_i^{} \cup E_j^{})) + 1 \ge \length(Z \cap D_i^{}) + \length(Z \cap E_j^{}).
\]
According to Proposition~\ref{T}(i), $T(k, l, l) = \emptyset$, hence $S(k, l, n) = \BN(k, l, (l, n))$,
and hence $S(k, l, n)$ is locally closed.
According to Proposition~\ref{S}(i), $S(k, l, l) = \emptyset$, hence $T(k, l, m) = \BN(k, l, (m, l))$,
and hence $T(k, l, m)$ is locally closed.
\end{proof}

\noindent
Combining Theorem~\ref{BN}, Proposition~\ref{S} and Proposition~\ref{T}
we can describe more explicitly the highest Brill-Noether locus.

\begin{corollary}
\label{highest}
We adopt the assumptions of Theorem~\ref{BN}.
We denote
\[
k_{\mathrm{max}}^{} = \max \{ l - m - 1, \, l - n - 1, \, 0 \}.
\]
For $k > k_{\mathrm{max}}^{}$, we claim that $\BN(k, l, (m, n)) = \emptyset$.
If $k_{\mathrm{max}}^{} > 0$, then we claim that
\[
\BN(k_{\mathrm{max}}^{}, l, (m, n)) =
\begin{cases}
\Ver(l, l) & \text{if $m > n$}, \\
\Hor(l, l) & \text{if $m < n$}, \\
\Ver(l, l) \sqcup \Hor(l, l) & \text{if $m = n$}.
\end{cases}
\]
\end{corollary}

\begin{corollary}
\label{BNDE}
We adopt the assumptions of Theorem~\ref{BN}.
We further assume that $l \le 2 m + 3$ and $l \le 2 n + 3$.
Then, for every $k \ge 1$,
\[
\BN(k, l, (m, n)) = \Ver(k + n + 1, l) \sqcup \Hor(k + m + 1, l).
\]
\end{corollary}

\section{The Euler characteristic of the flag Hilbert schemes}
\label{euler}

\noindent
For a scheme $S$ of finite type over $\CC$, we denote by $\euler(S)$ the topological Euler characteristic
of the complex analytic space associated to $S$.
Using straightforward arguments, we reduce the computation of $\euler(\Hilb(l, (m, n)))$
to the computation of the Euler characteristic of the Brill-Noether loci occurring in decomposition~\eqref{decomposition}.

\begin{proposition}
\label{BNHE}
Let $l \ge 1$, $m \ge 0$ and $n \ge 0$ be integers such that $m$ and $n$ are not both zero.
We claim that
\[
\euler(\Hilb(l, (m, n))) = (mn + m + n + 1 - l)\euler(\Hilb(l)) + \sum_{k \ge 1} k \euler(\BN(k, l, (m, n))).
\]
\end{proposition}

\begin{proof}
Consider $Z \in \Hilb(l)$.
From the short exact sequence
\[
0 \lra \I_Z^{}(m, n) \lra \Osh(m, n) \lra \Osh_Z^{} \lra 0
\]
we obtain the relation
\[
\dim_\CC^{} \HH^0(\I_Z^{}(m, n)) = \dim_\CC^{} \HH^1(\I_Z^{}(m, n)) + m n + m + n + 1 - l.
\]
Let $\phi \colon \Hilb(l, (m, n)) \to \Hilb(l)$ be the forgetful morphism.
For $Z \in \BN(k, l, (m, n))$ we have the relation
\[
\euler(\phi^{-1}(Z)) = \dim_\CC^{} \HH^0(\I_Z^{}(m, n)) = k + m n + m + n + 1 - l.
\]
Taking into account the decomposition~\eqref{decomposition}, we calculate:
\begin{align*}
\euler & (\Hilb(l, (m, n))) = \sum_{k \ge 0} (k + m n + m + n + 1 - l) \euler(\BN(k, l, (m, n))) \\
& = (m n + m + n + 1 - l)\sum_{k \ge 0} \euler(\BN(k, l, (m, n))) + \sum_{k \ge 0} k \euler(\BN(k, l, (m, n))) \\
& = (m n + m + n + 1 - l)\euler(\Hilb(l)) + \sum_{k \ge 1} k \euler(\BN(k, l, (m, n))). \qedhere
\end{align*}
\end{proof}

\noindent
Let $E \subset \PP^1 \times \PP^1$ be a line of bidegree $(0, 1)$.
Let $e \ge 1$ be an integer and let $\epsilon = (e_1^{}, \dots, e_h^{})$ be a partition of $e$.
Consider the discriminant locus
\[
\Delta(\epsilon) = \{ e_1^{} p_1^{} + \dots + e_{h_{}}^{} p_h^{} \mid \text{$p_i^{} \in E$ are mutually distinct} \}
\subset | \Osh_E^{}(e, 0) | \simeq \PP^{}{}^e.
\]
Consider the open subvariety
\[
U_h^{} = \{ (p_1^{}, \dots, p_h^{}) \mid \text{$p_i^{} \in E$ are mutually distinct} \} \subset E^h.
\]
The map $U_h^{} \to \Delta(\epsilon)$ given by $(p_1^{}, \dots, p_h^{}) \mapsto e_1^{} p_1^{} + \dots + e_{h_{}}^{} p_h^{}$
is a geometric quotient modulo the action of a certain subgroup $\Sigma_\epsilon^{}$
contained in the group of permutations of $h$ elements.
For $h \ge 3 $ we have $\euler(U_h^{}) = 0$.
We obtain the formulas
\begin{alignat*}{2}
\euler(\Delta(\epsilon)) & = \frac{1}{| \Sigma_\epsilon^{} |} \euler(U_h^{}) = 0 & \qquad & \text{if $h \ge 3$}, \\
\euler(\Delta(\epsilon)) & = \euler(U_2^{}) = 2 && \text{if $\epsilon = (e_1^{}, e_2^{})$ with $e_1^{} > e_2^{}$}, \\
\euler(\Delta(\epsilon)) & = \frac{1}{2} \euler(U_2^{}) = 1 && \text{if $\epsilon = (e_1^{}, e_2^{})$ with $e_1^{} = e_2^{}$}, \\
\euler(\Delta(\epsilon)) & = \euler(U_1^{}) = 2 && \text{if $\epsilon = (e)$}.
\end{alignat*}
Given a partition $\tau = (t_0^{}, \dots, t_{\nu - 1}^{}) \in \Pi(l)$ we consider the locally closed subset
\[
\Hilb_{\nu E}^{}(\tau) = \{ Z \mid Z \subset \nu E,\ \tau(Z) = \tau \} \subset \Hilb(l).
\]
The number $\bip(\tau)$ of bipartitions of $\tau$ is the number of pairs $(\tau', \tau'')$ such that $\tau = \tau' + \tau''$,
where $\tau'$ and $\tau''$ are non-increasing strings of length $\nu$ of non-negative integers.
Notice that
\[
\bip(\tau) = (t_0^{} - t_1^{} + 1) \cdots (t_{\nu - 2}^{} - t_{\nu - 1}^{} + 1) (t_{\nu - 1}^{} + 1).
\]

\begin{proposition}
\label{bipartitions}
We adopt the above notations.
We claim that
\[
\euler(\Hilb_{\nu E}^{}(\tau)) = \bip(\tau).
\]
\end{proposition}

\begin{proof}
Consider the morphism of schemes
\[
\psi \colon \Hilb_{\nu E}^{}(\tau) \lra \Hilb_E^{}(t_0^{}) \simeq \PP^{}{}^{t_0}
\]
given on closed points by $\psi(Z) = Z \cap E$.
Consider $Z_0^{} = e_1^{} p_1^{} + \dots + e_{h_{}}^{} p_h^{}$ in $\Hilb_E^{}(t_0^{})$.
Here $\epsilon = (e_1^{}, \dots, e_h^{})$ is a partition of $t_0^{}$
and $p_1^{}, \dots, p_h^{}$ are mutually distinct points of $E$.
Then $\psi(Z) = Z_0^{}$ if and only if $Z = P_1^{} \cup \dots \cup P_h^{}$,
where $P_i^{}$ is a subscheme of $\nu E$ that is concentrated at $p_i^{}$,
such that $\tau(P_i^{})_0^{} = e_i^{}$ and $\tau(P_1^{}) + \dots + \tau(P_h^{}) = \tau$.
According to Proposition~\ref{affine}, for a fixed partition $\tau_i^{}$,
the family of schemes $P_i^{} \subset \nu E$ that are concentrated at $p_i^{}$ and such that $\tau(P_i^{}) = \tau_i^{}$
is parametrized by an affine space $A(\tau_i^{})$ of dimension $\sum_{j \ge 1} \tau_{ij}^{}$.
It follows that
\[
\psi^{-1}(Z_0^{})
= \bigsqcup_{\substack{\tau_1 + \dots + \tau_h = \tau \\ \tau_{10} = e_1, \dots, \tau_{h0} = e_h}}
A(\tau_1^{}) \times \dots \times A(\tau_h^{}),
\]
hence
\[
\euler(\psi^{-1}(Z_0^{})) = \Big\vert \big\{ (\tau_1^{}, \dots, \tau_h^{})
\mid \tau_1^{} + \dots + \tau_h^{} = \tau, \ \tau_{10}^{} = e_1^{}, \dots, \ \tau_{h0}^{} = e_h^{} \big\} \Big\vert.
\]
The r.h.s.\ is constant when $Z_0^{}$ varies in the discriminant locus $\Delta(\epsilon) \subset \PP^{}{}^{t_0^{}}$.
We obtain the formulas
\begin{align*}
\euler & (\Hilb_{\nu E}^{}(\tau)) \\
& = \sum_{\epsilon \in \Pi(t_0)} \euler(\Delta(\epsilon))
\Big\vert \big\{ (\tau_1^{}, \dots, \tau_h^{})
\mid \tau_1^{} + \dots + \tau_h^{} = \tau, \ \tau_{10}^{} = e_1^{}, \dots, \ \tau_{h0}^{} = e_h^{} \big\} \Big\vert \\
& = \sum_{\substack{(e_1, e_2) \in \Pi(t_0) \\ e_1 > e_2}}
2 \Big\vert \big\{ (\tau_1^{}, \tau_2^{})
\mid \tau_1^{} + \tau_2^{} = \tau, \ \tau_{10}^{} = e_1^{}, \ \tau_{20}^{} = e_2^{} \big\} \Big\vert \\
& \qquad + \Big\vert \big\{ (\tau_1^{}, \tau_2^{})
\mid \tau_1^{} + \tau_2^{} = \tau, \ \tau_{10}^{} = \tau_{20}^{} \big\} \Big\vert + 2 \\
& = \Big\vert \big\{ (\tau_1^{}, \tau_2^{}) \mid \tau_1^{} + \tau_2^{} = \tau, \ \tau_1^{} \neq 0, \ \tau_2^{} \neq 0 \big\} \Big\vert + 2
= \bip(\tau). \qedhere
\end{align*}
\end{proof}

\noindent
Given positive integers $k$, $l$, $\lambda$, $\nu_1^{}, \dots, \nu_\lambda^{}$ and a non-negative integer $m$,
we define the sets
\begin{align*}
\Theta(k, l, m, \nu_1^{}, \dots, \nu_\lambda^{})
& = \Big\{ (\rho_1^{}, \dots, \rho_\lambda^{}) \mid \rho_i^{} = (r_{i,0}^{}, \dots, r_{i, \nu_i^{} - 1}^{}), \\
& \phantom{= \Big\{ (} r_{i,0}^{} \ge \dots \ge r_{i, \nu_i^{} - 1}^{} > 0, \ r_{ij}^{} \in \ZZ, \\
& \phantom{= \Big\{ (} l = \sum_{1 \le i \le \lambda} \ \sum_{0 \le j \le \nu_i - 1} r_{ij}^{}, \\
& \phantom{= \Big\{ (} k = \sum_{1 \le i \le \lambda} \ \sum_{\substack{0 \le j \le \nu_i - 1 \\ r_{ij} > m + 1}} (r_{ij} - m - 1) \Big\}, \\
\Phi(k, l, m) & = \bigcup_{\nu_1 \ge 1} \Theta(k, l, m, \nu_1^{}), \\
\Psi(k, l, m) & = \bigcup_{\nu_1 \ge 1, \, \nu_2 \ge 1} \Theta(k, l, m, \nu_1^{}, \nu_2^{}).
\end{align*}

\begin{proposition}
\label{STE}
Consider integers $k$, $l \ge 1$ and $m$, $n \ge 0$.
We claim that
\begin{align*}
\euler(S(k, l, n))
& = \sum_{\sigma \in \Phi(k, l, n)} 2 \bip(\sigma)
+ \sum_{(\sigma_1, \sigma_2) \in \Psi(k, l, n)} \bip(\sigma_1^{}) \bip(\sigma_2^{}), \\
\euler(T(k, l, m))
& = \sum_{\tau \in \Phi(k, l, m)} 2 \bip(\tau) + \sum_{(\tau_1, \tau_2) \in \Psi(k, l, m)} \bip(\tau_1^{}) \bip(\tau_2^{}).
\end{align*}
\end{proposition}

\begin{proof}
The two formulas are analogous, so we prove only the second one.
Given $Z \in \Hilb(l)$, write $\nu = \nu(Z)$ and $N(Z) = \nu_1^{} E_1^{} \cup \dots \cup \nu_\lambda^{} E_\lambda^{}$,
where $E_1^{}, \dots, E_\lambda^{}$ are distinct lines of bidegree $(0, 1)$ and $(\nu_1^{}, \dots, \nu_\lambda^{}) \in \Pi(\nu)$.
Let $Y_j^{}$ be the subscheme of $Z$ whose reduced support is contained in $E_j^{}$.
The condition that $Z$ belong to $T(k, l, m)$ is equivalent to the condition
\[
(\tau(Y_1^{}), \dots, \tau(Y_\lambda^{})) \in \Theta(k, l, m, \nu_1^{}, \dots, \nu_\lambda^{}).
\]
The subset of $T(k, l, m)$ with fixed $N(Z)$ and with fixed
\[
(\tau(Y_1^{}), \dots, \tau(Y_\lambda^{})) = (\tau_1^{}, \dots, \tau_\lambda^{}) \in \Theta(k, l, m, \nu_1^{}, \dots, \nu_\lambda^{})
\]
is parametrized by $\Hilb_{\nu_1 E_1}^{}(\tau_1^{}) \times \dots \times \Hilb_{\nu_\lambda E_\lambda}^{}(\tau_\lambda^{})$.
Applying Proposition~\ref{bipartitions}, we calculate:
\begin{align*}
\euler(T(k, l, m)) & = \sum_{\nu \ge 1} \ \sum_{(\nu_1, \dots, \nu_\lambda) \in \Pi(\nu)} \euler(\Delta(\nu_1^{}, \dots, \nu_\lambda^{}))
\!\!\!\!\!\!\!\!\!\!\!\!\!\!\!\!\!\!\!\!\!\!\!\! \sum_{(\tau_1, \dots, \tau_\lambda) \in \Theta(k, l, m, \nu_1, \dots, \nu_\lambda)}
\!\!\!\!\!\!\!\!\!\!\!\!\!\!\!\! \bip(\tau_1^{}) \cdots \bip(\tau_\lambda^{}) \\
& = \sum_{\nu \ge 1} \ \sum_{(\nu_1) \in \Pi(\nu)} \ \sum_{\tau_1 \in \Theta(k, l, m, \nu_1)} 2 \bip(\tau_1^{}) \\
& \phantom{= {}} + \sum_{\nu \ge 1} \ \sum_{\substack{(\nu_1, \nu_2) \in \Pi(\nu) \\ \nu_1 > \nu_2}} \
\sum_{(\tau_1, \tau_2) \in \Theta(k, l, m, \nu_1, \nu_2)} 2 \bip(\tau_1^{}) \bip(\tau_2^{}) \\
& \phantom{= {}} + \sum_{\nu \ge 1} \ \sum_{\substack{(\nu_1, \nu_2) \in \Pi(\nu) \\ \nu_1 = \nu_2}} \
\sum_{(\tau_1, \tau_2) \in \Theta(k, l, m, \nu_1, \nu_2)} \phantom{2} \bip(\tau_1^{}) \bip(\tau_2^{}) \\
& = \sum_{\tau \in \Phi(k, l, m)} 2 \bip(\tau) + \sum_{(\tau_1, \tau_2) \in \Psi(k, l, m)} \bip(\tau_1^{}) \bip(\tau_2^{}).
\qedhere
\end{align*}
\end{proof}

\noindent
Given integers $l \ge 1$ and $m \ge 0$, we write
\[
\csi(l, m) = \sum_{1 \le k \le l - m - 1}
\Big(\sum_{\rho \in \Phi(k, l, m)} 2 k \bip(\rho) + \sum_{(\rho_1, \rho_2) \in \Psi(k, l, m)} k \bip(\rho_1^{}) \bip(\rho_2^{})\Big).
\]
Notice that $\csi(l, m) = 0$ if $m \ge l - 1$.
For the last theorem we combine Proposition~\ref{BNHE}, Theorem~\ref{BN}, Proposition~\ref{STE},
Proposition~\ref{S}(i) and Proposition~\ref{T}(i).

\begin{theorem}
\label{HE}
Let $l \ge 2$, $m \ge 0$ and $n \ge 0$ be integers such that $m + n \ge l - 1$.
We claim that
\[
\euler(\Hilb(l, (m, n))) = (mn + m + n + 1 - l)\euler(\Hilb(l)) + \csi(l, m) + \csi(l, n).
\]
\end{theorem}

\noindent
At the end, we illustrate the above theorem for schemes of length at most $8$.
From \cite[Theorem (5.1)]{ellingsrud} or \cite[Theorem 0.1]{goettsche} we read the values of $\euler(\Hilb(l))$,
see Table~\ref{T1}.
The values of $\csi(l, m)$ are easily computed using the definition and are indicated in Table~\ref{T2}.
Substituting these into the above formula,
yields the values given in Table~\ref{T3} and Table~\ref{T4}.

\begin{table}[h]
\begin{center}
\caption{\small The values of $\euler(\Hilb(l))$ for $2 \le l \le 8$.}
\label{T1}
\begin{tabular}{| c | r | r | r | r | r | r | r |}
\hline
$l$ & $2$ & $3$ & $4$ & $5$ & $6$ & $7$ & $8$ \\
\hline
$\euler(\Hilb(l))$ & $\phantom{00}14$ & $\phantom{00}40$ & $\phantom{0}105$ & $\phantom{0}252$ & $\phantom{0}574$ & $1240$ & $2580$ \\
\hline
\end{tabular}
\end{center}
\end{table}

\begin{table}[h]
\begin{center}
\caption{\small The values of $\csi(l, m)$ for $2 \le l \le 8$ and $0 \le m \le l - 2$.}
\label{T2}
\begin{tabular}{| c | r || c | r || c | r || c | r |}
\hline
$(l, m)$ & $\csi(l, m)$ & $(l, m)$ & $\csi(l, m)$ & $(l, m)$ & $\csi(l, m)$ & $(l, m)$ & $\csi(l, m)$ \\
\hline
$(2, 0)$ & $6$ &
$(3, 0)$ & $36$ &
$(3, 1)$ & $8$ &
$(4, 0)$ & $152$ \\
$(4, 1)$ & $48$ &
$(4, 2)$ & $10$ &
$(5, 0)$ & $508$ &
$(5, 1)$ & $160$ \\
$(5, 2)$ & $60$ &
$(5, 3)$ & $12$ &
$(6, 0)$ & $1506$ &
$(6, 1)$ & $652$ \\
$(6, 2)$ & $246$ &
$(6, 3)$ & $72$ &
$(6, 4)$ & $14$ &
$(7, 0)$ & $4024$ \\
$(7, 1)$ & $1896$ &
$(7, 2)$ & $812$ &
$(7, 3)$ & $296$ &
$(7, 4)$ & $84$ \\
$(7, 5)$ & $16$ &
$(8, 0)$ & $10034$ &
$(8, 1)$ & $5024$ &
$(8, 2)$ & $2358$ \\
$(8, 3)$ & $980$ &
$(8, 4)$ & $346$ &
$(8, 5)$ & $96$ &
$(8, 6)$ & $18$ \\
\hline
\end{tabular}
\end{center}
\end{table}

{\small
\begin{table}[h]
\begin{center}
\caption{\small The values of $\euler(\Hilb(l, (m, n)))$ for $3 \le l \le 8$, $m + n \ge l - 1$ and $1 \le m \le n \le l - 2$.}
\label{T3}
\begin{tabular}{| c | r || c | r || c | r |}
\hline
$(l, (m, n))$ & {\tiny $\euler(\Hilb(l, (m, n)))$} &
$(l, (m, n))$ & {\tiny $\euler(\Hilb(l, (m, n)))$} &
$(l, (m, n))$ & {\tiny $\euler(\Hilb(l, (m, n)))$} \\
\hline
$(3, (1, 1))$ & $56$ &
$(4, (1, 2))$ & $268$ &
$(4, (2, 2))$ & $545$ \\
$(5, (1, 3))$ & $928$ &
$(5, (2, 2))$ & $1128$ &
$(5, (2, 3))$ & $1836$ \\
$(5, (3, 3))$ & $2796$ &
$(6, (1, 4))$ & $2962$ &
$(6, (2, 3))$ & $3762$ \\
$(6, (2, 4))$ & $5426$ &
$(6, (3, 3))$ & $5884$ &
$(6, (3, 4))$ & $8122$ \\
$(6, (4, 4))$ & $10934$ &
$(7, (1, 5))$ & $8112$ &
$(7, (2, 4))$ & $10816$ \\
$(7, (2, 5))$ & $14468$ &
$(7, (3, 3))$ & $11752$ &
$(7, (3, 4))$ & $16500$ \\
$(7, (3, 5))$ & $21392$ &
$(7, (4, 4))$ & $22488$ &
$(7, (4, 5))$ & $28620$ \\
$(7, (5, 5))$ & $35992$ &
$(8, (1, 6))$ & $20522$ &
$(8, (2, 5))$ & $28254$ \\
$(8, (2, 6))$ & $35916$ &
$(8, (3, 4))$ & $32286$ &
$(8, (3, 5))$ & $42356$ \\
$(8, (3, 6))$ & $52598$ &
$(8, (4, 4))$ & $44552$ &
$(8, (4, 5))$ & $57202$ \\
$(8, (4, 6))$ & $70024$ &
$(8, (5, 5))$ & $72432$ &
$(8, (5, 6))$ & $87834$ \\
$(8, (6, 6))$ & $105816$ &
& &
& \\
\hline
\end{tabular}
\end{center}
\end{table}
}

\begin{table}[h]
\begin{center}
\caption{\small The values of $\euler(\Hilb(l, (m, n)))$ for $2 \le l \le 8$ and $0 \le m \le l - 2 < n$.}
\label{T4}
\begin{tabular}{| c | r || c | r |}
\hline
$(l, (m, n))$ & $\quad \euler(\Hilb(l, (m, n)))$ &
$(l, (m, n))$ & $\quad \euler(\Hilb(l, (m, n)))$ \\
\hline
$(2, (0, n))$ & $\phantom{000}14^{} n - \phantom{0000}8$ &
$(3, (0, n))$ & $\phantom{000}40^{} n - \phantom{000}44$ \\
$(3, (1, n))$ & $\phantom{000}80^{} n - \phantom{000}32$ &
$(4, (0, n))$ & $\phantom{00}105^{} n - \phantom{00}163$ \\
$(4, (1, n))$ & $\phantom{00}210^{} n - \phantom{00}162$ &
$(4, (2, n))$ & $\phantom{00}315^{} n - \phantom{000}95$ \\
$(5, (0, n))$ & $\phantom{00}252^{} n - \phantom{00}500$ &
$(5, (1, n))$ & $\phantom{00}504^{} n - \phantom{00}596$ \\
$(5, (2, n))$ & $\phantom{00}756^{} n - \phantom{00}444$ &
$(5, (3, n))$ & $\phantom{0}1008^{} n - \phantom{00}240$ \\
$(6, (0, n))$ & $\phantom{00}574^{} n - \phantom{0}1364$ &
$(6, (1, n))$ & $\phantom{0}1148^{} n - \phantom{0}1644$ \\
$(6, (2, n))$ & $\phantom{0}1722^{} n - \phantom{0}1476$ &
$(6, (3, n))$ & $\phantom{0}2296^{} n - \phantom{0}1076$ \\
$(6, (4, n))$ & $\phantom{0}2870^{} n - \phantom{00}560$ &
$(7, (0, n))$ & $\phantom{0}1240^{} n - \phantom{0}3416$ \\
$(7, (1, n))$ & $\phantom{0}2480^{} n - \phantom{0}4304$ &
$(7, (2, n))$ & $\phantom{0}3720^{} n - \phantom{0}4148$ \\
$(7, (3, n))$ & $\phantom{0}4960^{} n - \phantom{0}3424$ &
$(7, (4, n))$ & $\phantom{0}6200^{} n - \phantom{0}2396$ \\
$(7, (5, n))$ & $\phantom{0}7440^{} n - \phantom{0}1224$ &
$(8, (0, n))$ & $\phantom{0}2580^{} n - \phantom{0}8026$ \\
$(8, (1, n))$ & $\phantom{0}5160^{} n - 10456$ &
$(8, (2, n))$ & $\phantom{0}7740^{} n - 10542$ \\
$(8, (3, n))$ & $10320^{} n - \phantom{0}9340$ &
$(8, (4, n))$ & $12900^{} n - \phantom{0}7394$ \\
$(8, (5, n))$ & $15480^{} n - \phantom{0}5064$ &
$(8, (6, n))$ & $18060^{} n - \phantom{0}2562$ \\
\hline
\end{tabular}
\end{center}
\end{table}

\end{document}